\documentclass[11 pt]{article}
\usepackage{amsmath,amssymb, amsthm,color} 
\usepackage{mathrsfs}
\usepackage{graphicx}
\usepackage{verbatim}
\usepackage[format=plain,  justification=raggedright,singlelinecheck=false, margin=1cm]{caption}

\newtheorem{theorem}{Theorem}
\newtheorem{thm}[theorem] {Theorem}
\newtheorem{construction}[theorem]{Construction}
\newtheorem{lem}[theorem]{Lemma}

\newtheorem{definition}[theorem]{Definition}
\newtheorem*{com}{Comment}
\def\Tau{\mathcal{T} }

\def\GG{\mathcal{G}}

\DeclareMathOperator{\diam}{diam}

\title{On homometric sets in graphs}

\author{Maria  Axenovich\\ 
\small Department of Mathematics, \\
\small Iowa State University, \\
\small Ames, IA 50011, USA\\
\small \texttt{email: axenovic@iastate.edu} 
\thanks{This author's research partially supported by NSF grant DMS-0901008}
\and 
Lale \"Ozkahya\\
\small \.Istanbul Bilgi \"University,\\
\small Department of Mathematics, \\
\small Kurtulu\c{s} Deresi Cad. 47,\\
\small 34435 Dolapdere Beyo\u{g}lu \.Istanbul Turkey\\
\small and\\[-0.8ex]
\small Department of Mathematics, \\
\small Iowa State University, \\
\small Ames, IA 50011, USA\\
\small \texttt{email: ozkahya@illinoisalumni.org}}

\begin{document}
\maketitle

\renewcommand{\thefootnote}{\empty}

\abstract {
For a vertex set $S\subseteq V(G)$ in a graph $G$, 
the {\em distance multiset}, $D(S)$,  is the multiset of pairwise distances 
between  vertices of $S$ in $G$. 
Two vertex sets are called {\em homometric} if their distance multisets are identical. 
For a graph $G$, the largest integer $h$,  such that there are  two disjoint homometric 
sets of order $h$ in $G$,  is denoted by $h(G)$. 
We slightly improve the general bound on this parameter introduced by  Albertson, Pach and Young~\cite{APY}
and investigate it in more detail for trees and graphs of bounded diameter. In particular, we show that 
for any tree $T$ on $n$ vertices $h(T) \geq \sqrt[3]{n}$ and for any graph $G$ of fixed diameter $d$, 
$h(G) \geq cn^{1/ (2d-2)}$. 
}

\section{Introduction}
 
For a vertex set $S\subseteq V(G)$, 
the {\em distance multiset}, $D(S)$,  is the multiset of pairwise distances 
between vertices of $S$ in $G$ .
 We say that two vertex sets 
are {\em homometric} if their distance sets are identical. 
``How large could two disjoint homometric sets be in a graph?" was a question of Albertson, Pach and Young~\cite{APY}.
Formally,  for a graph $G$, the largest integer $h$, such that there are  two disjoint homometric sets $S_1, S_2$ in $G$ with 
$|S_1|=|S_2|=h$, is denoted by $h(G)$. 
For a family of graphs $\GG$, $h(\GG)$ denotes the largest 
value of $h$ such that for each graph $G$ in $\GG$, $h(G) \geq h$. 
Let $h(n)$ be $h(\GG_n)$, where $\GG_n$ is the set of all graphs on $n$ vertices.
In  other words, $h(n) = \min \{h(G): |V(G)|=n\}$.   Albertson, Pach and Young \cite{APY}  provided the most general bounds. 
\begin{thm} [\cite{APY}]
$\frac{c \log{n}}{\log{\log{n}}} < h(n) \le \frac{n}{4}$ for $n > 3$, and a constant $c$.
\end{thm}
 
It is an easy observation that $h(G) = \lfloor |V(G)|/2 \rfloor$ when $G$ is a path or $G$ is a cycle.
However, more is known. Note that the multisets of distances for a vertex subset of
a path corresponds to a multiset of pairwise differences between elements of a subset  of positive integers. 
We shall say that two subsets of integers are homometric if their multisets of pairwise differences coincide.
Among others, Rosenblatt and Seymour~\cite{RS} 
proved that  two multisets  $A$ and $B$ of integers  are homometric if and only if  
there are two multisets $U, V$ of integers such that  $A= U+V$ and $B=U-V$, 
where $U+V$ and $U-V$ are multisets, $U+V= \{u+v: ~ u\in U, ~ v\in V\}$,   $U-V= \{u-v: ~ u\in U, ~ v\in V\}$. 
Lemke, Skiena and Smith~\cite{LSS} showed that if $G$ is a cycle of length $2n$, 
then every subset of $V(G)$ with $n$ vertices and its complement are homometric sets. 
Suprisingly, when the class $\GG$  of graphs under consideration is not a path or a cycle, the problem of finding $h(\GG)$
becomes nontrivial. Even when $\GG$  is a  class of $n$-vertex graphs that are the unions of pairs of paths sharing a single point, $h(\GG)$ 
is not known.   Here, we use standard graph-theoretic terminology, see for example \cite{bollobas} or \cite{west}.

The homometric set problem we consider here has its origins in Euclidean geometry, 
with applications in X-ray crystallography introduced in the 1930's with later applications in restriction site mapping of DNA. 
In particular, the fundamental problem that was considered  is whether one could identify a given set of points from its multiset of distances.
There are several related directions of research in the area, for example the question of 
recognizing the multisets corresponding to  
a multiset  of distances realized by a set of points in the Euclidean space of given dimension, see ~\cite{matousek}.

In this paper,  we provide the new bounds on  $h(\GG)$  in terms of densities and diameter,  we also investigate this function for various classes 
of graphs, in particular for trees.
In Section \ref{Main-Results} we state our main results.  
We give the basic constructions of homometric sets 
in Section \ref{Constructions}.  
We provide the proofs  of the main results in Section~\ref{Proofs}.
Finally, in Section \ref{More-Trees},  we give additional bounds for $h(T)$, when $T$ is a tree in terms of its parameters.

 \section {Main Results} \label{Main-Results}
 
  Let  $\mathcal{T}_n$ the set of all trees on $n$ vertices. 
 A {\it spider} is a tree that is a union of vertex-disjoint paths, called {\it legs} and  
a vertex that is adjacent to one of the endpoints of each leg, called the {\it head}.  
Let $\mathcal{S}_{n,k}$ be the set of $n$-vertex spiders with $k$ legs and 
 $\mathcal{S}_n$ be the set of all $n$-vertex spiders.
A {\em caterpillar} is a tree,  that is a union of a path, called {\it spine}, and leaves adjacent to the spine.
Let  $\mathcal{R}_n$ be the set of all caterpillars on $n$ vertices.  
Finally, a {\it haircomb} is a tree, that consists of a path called the {\it spine} and a collection of vertex-disjoint 
paths, called legs,  that have an endpoint on the spine. 
Let $\mathcal{H}_n$ be the set of haircombs on $n$ vertices.
 Here, we slightly improve the known general bounds of $h(G)$,  provide new bounds on $h(G)$  in terms of 
 density and diameter, and give several results for $h(T)$ in case when $T$ is a tree.

\begin{thm}\label{upper-bd}
For infinitely many values of $n$, and a positive constant $c$, 
$h(n)\le n/4 - c\log{\log{n}}$.
\end{thm}

\begin{thm}\label{diam-kg}
Let $G$ be a graph 
on $n$ vertices with $e$ edges, diameter $d$, $n\geq 5$, $d\geq 2$.   
If for an integer $k$, $\binom{{k\choose 2}+d-1 }{ d-1} < n - 2k +2$, 
then $$h(G) \ge \max\{k,~d/2, ~\sqrt{2e/n}\}.$$ 
In particular, 
$$h(G)\ge \max\{0.5n^{1/(2d-2)},~d/2, ~\sqrt{2e/n}\}$$
for $n\ge c(d) = d^{2d-2}.$
\end{thm}

\noindent
Since it is well known, see \cite{bollobas}, that almost all graphs $G(n,p)$ have diameter $2$,  
the above theorem implies that for almost every graph $G=G(n,p)$,  $h(G_{n,p}) = \Omega(\sqrt{n})$.

In the following theorems, we omit ceilings and floors for simplicity.

\begin{thm}\label{trees-weaker-bd} 
For a positive integer $n$, $h(\mathcal{T}_n)\ge n^{1/3} -1$.
 \end{thm}

\begin{thm}\label{special}
For a positive integer $n$,   $h(\mathcal{R}_n) \geq n/6$, $h(\mathcal{H}_n) \geq \sqrt{n}/2$.
\end{thm}

\begin{thm}\label{spiders}
For positive integers $n$,  $k$,   such that $k<n$, 
\begin{equation}\nonumber
 h(\mathcal{S}_{n,k}) \ge 
\begin{cases}
\frac{5}{12}n,& k=3,\\
\frac {1}{3}n, &  k=4,\\
\left(\frac{1}{4} + \frac{3}{8k-12}\right)n, & k\ge 5. 
\end{cases}
\end{equation}
Moreover, $h(\mathcal{S}_{n,n/2}) = (n+2)/4$, and   $h(T) = n/2$ for any $T\in \mathcal{S}_{n,3}$ with 
legs on $l_1, l_2, l_3$ vertices if $l_1+1=t(l_2-l_3)$  for an odd integer $t\ge 1$ or if 
$l_1 = l_2 = l_3$. 
\end{thm}

 The following two theorems are corollaries of previously known number-theoretic and graph-theoretic results.

\begin{thm}\label{dense-case}
For every fixed $\alpha > 0$ and for every $\epsilon > 0$ there exists $N=N(\alpha, \epsilon)$ 
so that for all $n > N$, if $G$ is a graph on $n$ vertices, diameter $2$ and at most $n^{2-\alpha}$ edges or 
at least $\binom{n}{2} -  n^{2-\alpha}$ edges, then
$h(G) \geq n/2-\epsilon n$.  
\end{thm}

\begin{thm} \label{general-diam2}
Let $G$ be  a graph with diameter $2$ and  even number of vertices $n\geq 90$. Let  $V_i = \{v\in V: \deg(v)=i\}$. 
If $G$ satisfies one of the following conditions:\\
1) $|\{i: |V_i|\, \text{ is odd}\}|>n/2$,\\
2) $|V_i|$ is even for each $i$, \\
then $h(G) = n/2$.
\end{thm}

 \section{ Constructions and definitions}\label{Constructions}

In the remaining part of the paper, 
we may omit the floor and ceiling of fractions for simplicity. For a tree $T$,  and its vertex $r$, let   $P=P(T, r)$ be a partial order on the vertex set 
of a tree $T$, such that $x<y$ in $P$ if the $x$,$r$-path in $T$ contains $y$. 
We call this an $r$-{\it order} of $T$ and say that $T$ is $r$-ordered. 
A vertex $y$  is a {\it parent} of  $x$ in $P(T, r)$,   and  a vertex $x$ is a {\it child} of $y$  if $x < y$ and there is no other element $z$ 
such that $x < z < y$. If neither $x< y$ nor $y< x$ for two elements $x$ and $y$ in $P(T,r)$, 
then we say that $x$ and $y$ are {\it noncomparable}.  An {\it antichain} is defined as a set of pairwise noncomparable elements. 
A pair of vertices $x$ and $y$ are called {\it siblings} if they have the same parent.  \\

For a vertex $x$  of degree at least $3$ in a tree $T$, the connected components of  $T-x$   that are paths are called {\it pendent } paths of $x$. 
The endpoints of these paths adjacent to $x$ in $T$ are called {\it attachment} vertices.
A vertex $x$ of degree at least $3$ is called {\it bad} if it has an odd number of pendent paths.
Moreover, let a shortest pendent path corresponding to a bad vertex $x$ be called a {\it bad path}.\\

We call a tree $T'$ a {\it cleaned} $T$ if $T'$ is obtained from $T$ 
by removing the  vertices of all bad paths. 
A tree $T''$ is called {\it trimmed} $T$ is it is obtained from $T$ by removing the
vertices of  all bad paths and  removing  all the vertices except for the attachment vertices of all remaining pendent paths of $T$.\\

Let $bad(T)$, $bad_3(T)$ be the number of bad vertices, and the number of bad vertices of degree $3$ in $T$, respectively. 
Let $bad_l(T)$ be the total number of vertices in bad paths of $T$.
Let $N_i(x)$ be the set of vertices at distance $i$ from a vertex $x$. 
For two disjoint sets $A',A''\subseteq V(G)$, 
let $D(A',A'')$ be the multiset 
of distances in $G$ between pairs $u'$ and $u''$, $u'\in A'$, $u''\in A''$.  \\

  \begin{construction} \cite{APY}\label{c-apy}
  Let $G$ be a graph and $G'=(v_1, \ldots, v_{2t})$ be a shortest $v_1$,$v_{2t}$-path in $G$, for some $t\geq 1$.
  Let $S_1= \{v_1, \ldots,  v_ t\}$, $ S_2  = \{  v_{ t +1}, \ldots, v_{2t}\}$.
  Since both $S_1$ and $S_2$ induce shortest paths of the same length, they form homometric sets. 
  \end{construction}
  
  \begin{construction}\label{balanced-ptn}
  Let $T$ be a tree and $S$ be an antichain in $P(T,r)$  such that each vertex in $S$, 
has a  sibling  in $S$.  Let   $S'_1, \ldots, S'_k$ be the maximal families of siblings in $S$. 
Let $S'_i = A_i\cup B_i\cup C_i$, where $A_i, B_i, C_i$ are disjoint, $|A_i|=|B_i|$, $|C_i| \leq 1$, $i=1, \ldots, k$.
Let $S_1= A_1\cup A_2 \cup \cdots \cup A_k$, $S_2= B_1\cup B_2\cup \cdots \cup B_k$. 
The sets $S_1$ and $S_2$ are homometric. See Figure~\ref{hom-set}.  
\end{construction}

\begin{figure}[h]
\begin{center}
\vspace{-0.4cm}
\includegraphics[scale=0.4]{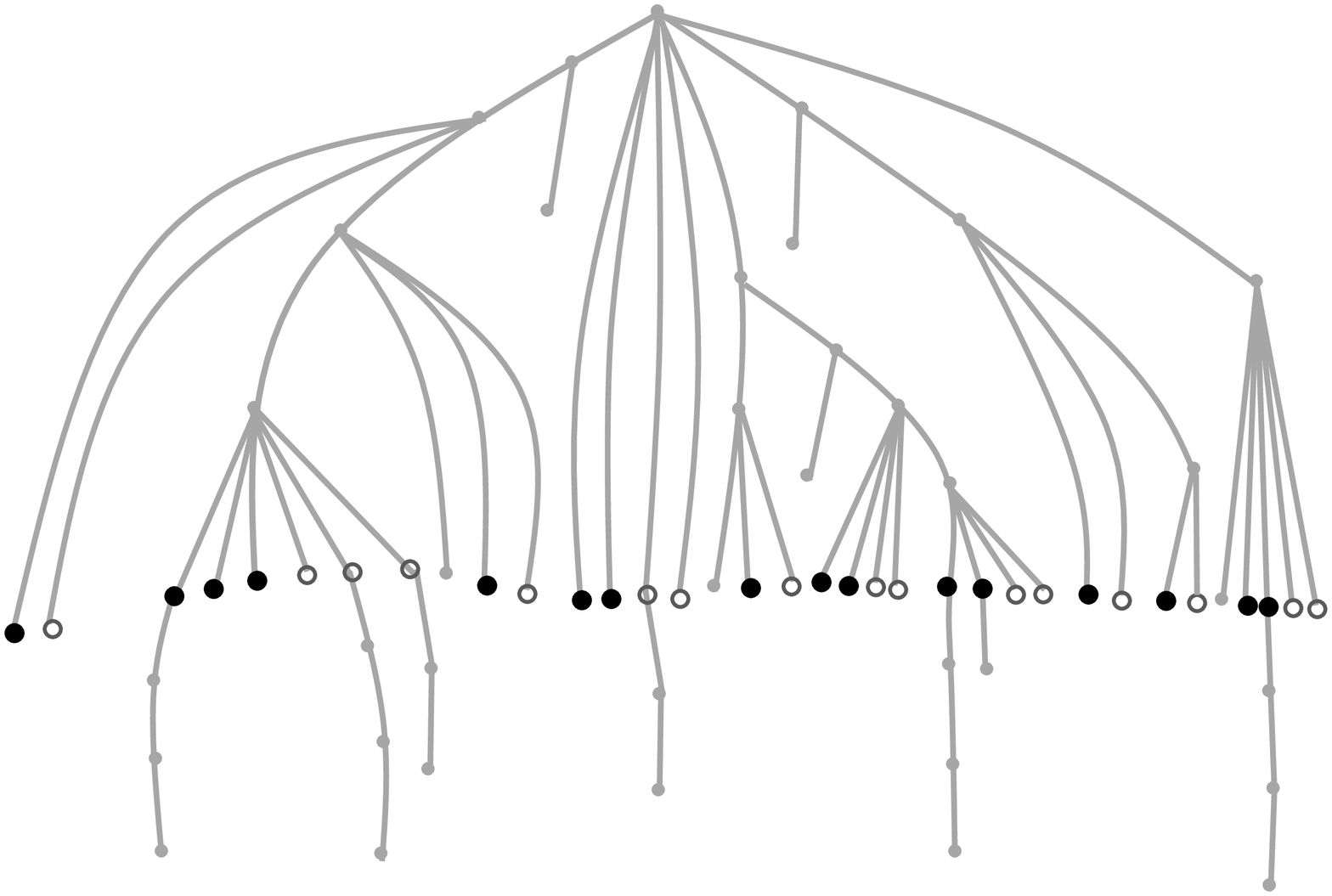} 
\vspace{-.8cm}
\end{center}
\caption{{\small Example of Construction~\ref{balanced-ptn} 
with $S_1$ and $S_2$ consisting of black and hollow vertices, respectively. }}
\label{hom-set}
\end{figure}

\begin{construction}\label{three-legs} 
Let $T\in \mathcal{S}_{n,3}$ with head $v$ and legs $L_1, L_2, L_3$ 
on $l_1, l_2, l_3$ vertices, 
respectively, $l_1\ge l_2.$ 
Let $L_3$ be a path $v_1,\dots,v_{l_3}$, 
where $v_{l_3}$ is a leaf of $T$. 
Let $S_1$ be the set consisting of $v$, the vertices 
$v_{2i}$, $1\le i\le \lfloor l_3/2\rfloor$ and 
$l_2-1$ vertices in $L_1$ closest to $v$. 
Let $S_2$ be the set consisting of the vertices 
$v_{2i-1}$, $1\le i\le \lfloor l_3/2\rfloor$  
and $V(L_2)$. 
The sets $S_1$ and $S_2$ are homometric 
(see Figure~\ref{spider}(b)).
\end{construction}

\begin{construction}\label{three-legs-special}  
Let $T\in \mathcal{S}_{n,3}$ with head $v=v_0$ and legs $L_1, L_2, L_3$ on 
$\ell_1, \ell_2, \ell_3$ vertices, respectively, $l_1>l_2.$  
Let $L_3$ be a path $v_1,\dots,v_{l_3}$, 
where $v_{l_3}$ is a leaf of $T$. Let $l_1 - l_2 = x>0$, $l_3+1 = bx + r$, where $b$ and $r$ are integers, $0 \leq r  < x$.  
If $b$ is odd, let $a=b-1$ and if $b$ is even, let $a=b$. 
Let  $P_i = \{v_{ix},\dots,v_{ix+x-1}\}$ for $0 \le i \le a$. 
Let $S_1$ be the union of  $P_{2i}$, $0\le i\le a/2$ and $V(L_2)$. 
Let $S_2$ be the union of $P_{2i-1}$, $1\le i\le a/2$ and $V(L_1)$. 
The sets $S_1$ and $S_2$ are homometric 
(see Figure~\ref{spider}(c)).
\end{construction}

Construction~\ref{three-legs-special} is not used in any of the proofs here. 
However, it hints that for a tree $T\in \mathcal{S}_{n,3}$, 
$h(T)$ can be very close to $n/2$,  
depending on the optimized value of $x$. 

\begin{construction} \label{k-legs} 
Let $T$ be a spider with head $v$ and $k$ legs 
$L_1,L_2,\dots,L_k$, $k\ge 3$, 
on $\ell_1 \geq \ell_2 \geq \cdots \geq \ell_k$ vertices, respectively. 
Let  $A_i =V(L_{2i})$, let 
$B_i$ be the set of $|A_i|$ vertices closest to $v$ in $L_{2i-1}$ , $i= 1, \ldots, \lfloor k/2 \rfloor=m$.
Then let $S_1 = A_1 \cup A_2\cup \cdots \cup A_m$, $S_2 = B_1 \cup B_2\cup \cdots \cup B_m$.
Clearly $S_1$ and $S_2$ are homometric sets (see Figure~\ref{spider}(d)).
\end{construction}

\begin{figure}[h]
\vspace{-.5cm}
\begin{center}
\centering
\includegraphics[scale=0.4]{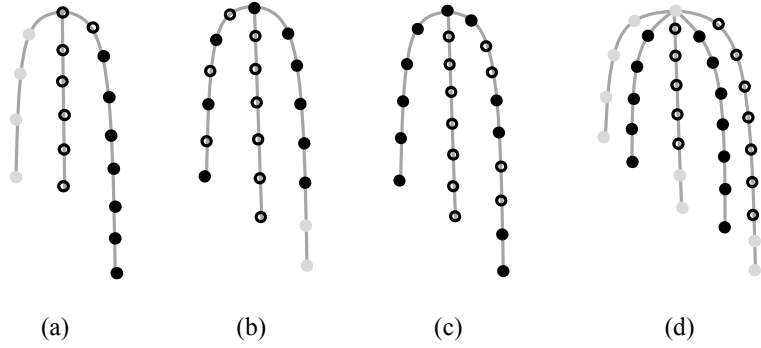}
\end{center}
\vspace{-3.cm}
\caption{{\small Examples of homometric sets in spiders using  Constructions~\ref{c-apy}, \ref{three-legs}, \ref{three-legs-special} (with $x=2$, $a=4$) and \ref{k-legs}, respectively.  
Two homometric sets consist of  black and hollow vertices, respectively. }}
\label{spider} 
\end{figure}

We shall say that two integers are {\it almost equal} if they differ by $1, 0$, or $-1$.~
\begin{thm}[Karolyi~\cite{karolyi}]\label{kar}
Let $X$ be a set of $m$ integers, each between $1$ and $2m-2$.  If $m\geq 89$, then one can partition $X$ into two sets, $X_1$ and $X_2$  of almost equal sizes such that 
the sum of elements in $X_1$ is almost equal to the sum of elements in $X_2$.
\end{thm}

\begin{thm}[Caro, Yuster~\cite{caroyuster}]\label{num-edges}
For every fixed $\alpha > 0$ and for every $\epsilon > 0$ there exists $N=N(\alpha, \epsilon)$ 
so that for all $n > N$, if $G$ is a graph on $n$ vertices and at most $n^{2-\alpha}$ edges 
then there are two vertex  disjoint subgraphs of the same order and size with at least 
$n/2-\epsilon n$ vertices in each of them.  
\end{thm}




\section{Proofs of main results} \label{Proofs}  

\begin{definition}
For a graph $H$, we say that $G$ is an {\bf $(H, m, v)$-flower}  with a path $P$ if
 $G$  is a vertex-disjoint union of $H$ and an $m$-vertex  path $P$ with endpoint $v$,
 together with all edges between $V(H)$ and $v$.
\end{definition}

  \begin{lem}\label{flower}
 Let $G$ be an $(H, m, v)$-flower with a path $P$. 
If $S_1$ and $S_2$ are homometric sets of $G$, where $|S_1|=|S_2| \geq 2$, 
 then  either $S_1\cup S_2 \subseteq V(P) \cup \{u\}$, $u \in V(H)$ or
 $S_1\cup S_2 \subseteq V(H) \cup \{v, v_1\}$, where $v_1$ is the neighbor of $v$ in $P$. 
  \end{lem}
  
  \begin{proof}
Let $P$ be a path of a flower.  Let $S_1, S_2$ be homometric sets in $G$ of size at least $2$ each.  
Assume that $(S_1\cup S_2)  \cap V(P) \neq \emptyset$ and $(S_1\cup S_2)  \cap V(H) \neq \emptyset$.
Note that $D(V(H))$ consists of $1$s and $2$s.

Consider the vertex $x\in (S_1\cup S_2)\cap V(P) $,  farthest  from $v$,  at a distance at least $2$ from $v$.
Without loss of generality,  $x\in S_1$.  If there is $x'\in S_1\cap V(H)$ then there is no pair of vertices in $S_2$ 
with the distance equal to the distance between $x$ and $x'$. Thus, $S_1\cap V(H)= \emptyset$.
Let $a$ be the largest distance in $D(S_1)$, $a\geq 3$.
Let $y \in S_2 \cap V(H) $, let $y'$ be at a distance $a$ from $y$, $y \in V(P)$. 
Note that $a$ appears exactly once in $D(S_1)$ and thus it appears exactly once in $D(S_2)$. 
Since all vertices from $V(H)$ are at the same distance from $y$, there is exactly one vertex $y' \in S_2 \cap V(H)$.  
  \end{proof}

\begin{proof}[Proof of Theorem~\ref{upper-bd}]

For a fixed integer $k$,  
let $a_0, a_1,\ldots, a_k$ be a  sequence  of integers such that 
$a_0 = 1$,  $a_1\geq  5$ and each $a_i$, $i\ge 2$, is the smallest odd number satisfying 
$$
a_i > 4 \left( 1+\sum_{j=1}^{i-1}{a_j+1\choose 2}\right).
$$
Let  $n=  2(a_1+\dots +a_k) - k/4$. 
Let $H$ be a  vertex-disjoint union of cliques on $a_1, a_2, \ldots, a_k$ vertices, respectively. 
Let the vertex sets of these cliques be $Q_1, Q_2, \ldots, Q_k$, respectively.
Note that $a_j> 4\sum_{i=0}^{j-1}a_i$ for $j\ge 1$, which implies that 
\begin{equation}\label{vx-sum}
a_j\ge \frac{4}{5}(\sum_{i=0}^j a_i).
\end{equation} 

Let $G$ be an $(H, n/2 - k/8,  v)$-flower with a path $P$. 
Note that $k = c \log\log n$, for a constant $c$.
This construction of $G$ is inspired by an example given by Caro and Yuster in~\cite{caroyuster}.  \\

Let $S_1$, $S_2$ be largest homometric sets in $G$.  By Lemma \ref{flower}, 
$S_1\cup S_2 \subseteq V(P) \cup \{u\}$, $u \in V(H)$ or
 $S_1\cup S_2 \subseteq V(H) \cup \{v, v_1\}$, where $v_1$ is the neighbor of $v$ in $P$. 
~\\
 
{\it Case 1~~}  $S_1\cup S_2 \subseteq V(P) \cup \{u\}$, $u \in V(H)$. \\
 Then $h(G) \leq  (|V(P)|+1)/2 = n/4-k/16 +1/2 = n/ 4 - c\log\log n$,  for a positive constant $c$.\\
 
 {\it Case 2~~}  $S_1\cup S_2 \subseteq V(H) \cup \{v, v_1\}$, where $v_1$ is the neighbor of $v$ in $P$.\\
  Let $\{v_1\} = Q_0$, then $S_1\cup S_2$ is a vertex subset of a join of vertex-disjoint cliques on sets $Q_0, \ldots, Q_k$ and $\{v\}$. 
 Since $G[V(H) \cup \{v, v_1\}]$ is a graph of diameter $2$, $S_1$ and $S_2$ induce the same number of edges in $G$. 
For $i =0, \ldots, k$, let $ Q_i' = Q_i \cup \{v\}$ if $v\in S_1\cup S_2$, and let $Q_i'=Q_i$ if $v\not\in S_1\cup S_2$.

 If $v\not\in S_1\cup S_2$ and  for all $j$, $|S_1\cap Q'_j|= |S_2 \cap Q'_j|$, then, 
since $Q_j$ has odd size,  $ (S_1 \cup S_2) \cap Q_j  \neq Q_j$,  and 
  $h(G) \leq (|V(H)|-k)/2 = n/4  + k/16 - k/2 =  n/2 - c' \log\log n$, for a positive constant $c'$.
 If $v\in S_1\cup S_2$ and  for all $j$,  $|S_1\cap Q'_j| = |S_2 \cap Q'_j|$, then  $|S_1|\neq |S_2|$, since $v\in Q'_j$ for all $j$.
So,we can assume that there is $j$, for which $|S_1\cap Q'_j|\neq |S_2 \cap Q'_j|$. 
Let $j$ be the largest such index and  let $Q'_j=Q$.

Assume first that $j\le k/2-2$ and $ v\not \in S_1\cup S_2$.
For each $i>j$, $(S_1 \cup S_2) \cap Q_i \neq Q_i$. Thus $h(G)\le (n/2 + k/8 - k/2+2)/2 = n/4 - c \log\log n$,  for a positive constant $c$. 

Now assume that $j\le k/2-2$ and $ v \in S_1\cup S_2$.  Without loss of generality, let $v\in S_1$.
Let  ${\bf Q}=  Q_0 \cup \cdots \cup Q_{j}$,  ${\bf Q'}=  Q_{j+1} \cup \cdots \cup Q_{k}$.
Then $|S_2\cap {\bf Q'}|-|S_1\cap {\bf Q'} |  \ge k-j-1 \ge k/2+1$.
Since $|S_1\cap ({\bf Q}\cup  {\bf Q'}\cup \{v\}) |= |S_2\cap ({\bf Q}\cup  {\bf Q'}\cup \{v\}) |$,   \quad
$|{\bf Q}| \geq -|S_2\cap {\bf Q}| +  |S_1\cap {\bf Q}|  \geq  k/2.$ 
Therefore, \eqref{vx-sum} implies that 
$a_j/2\ge \frac{4}{10}|{\bf Q}| \geq  \frac{2}{5} (k/2)\ge k/5.$ \\
So, we have that either $a_j/2\ge k/5 $ or that $j\geq k/2+2$.  
In any case, we have that $a_j/2\ge k/5$. 

If $|Q- (S_1\cup S_2)| \geq a_j/2 $, 
then $h(G)\le (n/2 + k/8 - a_j/2 )/2\le (n/2 +k/8 - k/5)/2\leq n/4 - c\log\log n $, for a positive constant $c$. 
Otherwise,   $|Q\cap (S_1\cup S_2) | >a_j/2$.  
Then,  the number of edges induced by $S_1 \cap Q$ differs from the 
number of edges induced by $S_2 \cap Q$ by at least $a_j/4$. 
The number of edges induced by $S_1 \cap  ({\bf Q'}\cup \{v\})$ is 
the same as number of edges induced by $S_2 \cap   ({\bf Q'}\cup \{v\})$. 
The total number of edges induced by ${\bf Q}\cup \{v\}$ is at most 
$ 1+\sum_{i=1}^{j-1}{a_i+1\choose 2}< a_j/4$.  Thus $S_1$ and $S_2$ can not induce the same number of edges, a contradiction.
 \end{proof}

\begin{proof}[Proof of Theorem~\ref{diam-kg}]

The fact that $2h(G) \geq d$ follows from Construction \ref{c-apy}.   For positive integers $k,n$ ($k<n$),  the {\em Kneser graph} $KG(n,k)$ is a graph on the vertex set ${[n]\choose k}$ 
whose edge set consists of pairs of  disjoint $k$-sets. 
Lov\'asz~\cite{lovasz} proved that the chromatic number of the Kneser graph $KG(n,k)$ 
is $n-2k+2$.  
We fix $k\le n/2$ and consider the Kneser graph $\mathcal{K} = KG(n,k)$. 
Considering a graph $G$ with vertex set $[n]$, 
we define a coloring of $\mathcal{K}$ 
by letting the distance multiset of each $k$-subset of $V(G)$ be 
the color of the corresponding vertex in $\mathcal{K}$. 
Since any vertex pair in $G$ has distance in $\{1,\dots,d\}$, 
the number of possible colors that are used on $\mathcal{K}$ 
is at most ${{k \choose 2}+d-1 \choose d-1}$.   
If 
\begin{equation}\label{kn-aim}
{{k \choose 2}+d-1 \choose d-1} < n-2k+2,
\end{equation}
then  there are 
two adjacent vertices of the same color in $\mathcal{K}$ that 
correspond to a pair of disjoint $k$-subsets of $V(G)$ that are homometric. \\

\noindent
Note that when $d=2$, we have ${{k \choose 2}+d-1 \choose d-1} = {k \choose 2}+1< n -2k+2$
for  $k = \sqrt{n}$.\\

\noindent
Let  $v\in V(G)$ be a vertex of a maximum degree $\Delta(G)$. The closed neighborhood 
$N[v]$ induces a graph of diameter $2$. Moreover, for any two vertices in $N[v]$,  
the distance between them in $G$ is the same as in $G[N[v]]$.
As before, we see that  (\ref{kn-aim}) with $d=2$ and $n= \Delta$ holds for  $k \geq \sqrt{\Delta(G)}$.
Therefore,  $h(G[N[v]]) \geq \sqrt{\Delta(G)}\geq  \sqrt{(2e/n)+1}$. 
This proves the first part of the theorem.\\

\noindent
To prove the second statement of the theorem,  we  assume that $d\ge 3$ and  $n\ge d^{2d-2}$ and show 
 that  for any $k$, such that $d/2 < k\le 0.5n^{1/(2d-2)}$, the inequality  \eqref{kn-aim} holds.
Since $d-1 \le k^2/2$, \;$2^{d-2}\le (d-1)!$, and  $k\le 0.5n^{1/(2d-2)}\le n/4$,  we have that 
 $${{k \choose 2}+d-1 \choose d-1}<  \frac{\left({k \choose 2}+d-1\right)^{d-1}}{(d-1)!} \le    \frac{(k^2)^{d-1}}{(d-1)!  }\le \frac{(k^2)^{d-1}}{2^{d-2}}
  \le \frac{n}{2^{d-2}}\le \frac{n}{2}\le n-2k + 2. $$
Therefore, there are homometric sets of size $k$ for any $k$, 
$d/2\le k\le 0.5n^{1/(2d-2)}$. 
\end{proof}

 \begin{lem}\label{lem-main}
Let $T$ be a tree, $r$ be its vertex,  and $S$ be an antichain in $P(T,r)$  such that 
each vertex $ S$, has a  sibling  in $S$.  Let $k'=k'(S)$ 
be the number of maximal odd sets of siblings in $S$. 
Then 
$$
2h(T) \geq \max\left\{\diam(T)+1,~ |S|-k',~ \frac{2}{3}|S|,  ~
d_1(T) - bad(T), ~\frac{n}{\diam(T)} - bad(T), ~\frac{n-  bad_l(T)}{\diam(T)}\right\}.$$
\end{lem}
   
 \begin{proof}
The fact that $2h(T) \geq \diam(T)+1$ and $2h(T) \geq |S|-k'$ follows immediately from Constructions \ref{c-apy} and \ref{balanced-ptn}.
Since each maximal family of siblings in $S$ has at least two elements, 
$k' \leq |S|/3$, thus $2h(T) \geq 2|S|/3$. \\

\noindent
For any tree $\Tau$, Dilworth's theorem applied to $P(\Tau)$ implies that the size 
of a  maximum antichain in $P(\Tau)$  is  at least $n/ \diam(\Tau).$
The set of leaves is a maximum antichain of $P(\Tau)$, thus $d_1(\Tau) \ge n/\diam(\Tau).$\\

\noindent
Let $T'$ be trimmed $T$. 
Since $T'\subseteq T$, $h(T) \geq h(T')$. 
Let $S$ be the set of leaves of $T'$. Note that $|S|= d_1(T')= d_1(T)- bad(T)$.  Moreover, 
since $T'$ has no bad vertices, $k'(S)=0$ in $T'$.   So,  $2h(T) \ge |S| - k'(S) = d_1(T)  - bad(T) \ge {n}/{\diam(T)} - bad(T).$\\

\noindent
Let $T''$ be cleaned $T$. We have that 
$|V(T) \setminus V(T')| = bad_l(T)$, moreover, $T''$ has no bad vertices.
Thus $2h(T) \geq 2h(T'') \geq |d_1(T'')|\ge (n - bad_l (T))/\diam(T'')\geq (n-bad_l(T))/\diam(T)$.  
\end{proof}

\begin{proof}[Proof of Theorem~\ref{trees-weaker-bd}]
Assume that $v$ is a vertex contained in the center of $T$ and 
let $t = \lceil \diam(T)/2\rceil$.  Let $P=P(T, v)$.
We define a partition 
$V(T) = N_0 \cup N_1\cup N_2 \cup \dots \cup N_t$, 
where $N_i=N_i(v)$.  So, $N_i$ is an antichain in $P$, $i=1, \ldots, t$.
Since $h(T)\ge t$, we can assume that $t \leq n^{1/3} -1$, and thus  $\diam(T)\le 2n^{\frac{1}{3}}-2$. 
Let  $x$ be the smallest integer such that $|N_i| \le |N_{i-1}|+x-1$  for $1\le i \le t$. Then 
$n = 1+ \sum_{i=1}^t |N_i| \le 1+ \sum_{i=1}^t ix\leq   x (t+1)^2/2.$ 

Since $t \le n^{1/3}-1$, by the above inequality we have $x \ge 2n^{1/3} $. 
For some $j$, $1\le j \le t$, $|N_j| - |N_{j-1}| = x-1$.  
 Let $S$ be a largest  subset of $N_j$  such that each vertex in $S$ has a sibling in $S$. 
 Recall that  $k'(S)$ is  the number of maximal sets of siblings in $S$ of odd size, i.e., 
 $k'(S)$ is the number of vertices of $N_{j-1}$ with odd number of neighbors in $N_j$.
So, $|S|-k'(S) \geq |N_j|-|N_{j-1}| = x-1 \geq 2n^{1/3} -1$. 
By  Lemma~\ref{lem-main}, $h(T)  \ge  n^{1/3} - 1$.
\end{proof}

%
%
%
%
%

\begin{proof}[Proof of Theorem \ref{special}]
Let $T$ be a caterpillar on $n$ vertices and 
assume that its spine $P$ has at least $n/3$ vertices. 
Then, by Theorem \ref{diam-kg}, $h(T)\ge n/6$. 
If $P$ has less than $n/3$ vertices, then 
there are at least $2n/3$ leaves. 
Assume that there are exactly $k$ vertices on the spine with degree at least 4, 
label them as $v_1,\dots,v_k$.  
Apply Lemma \ref{lem-main} to a set $S$ consisting of leaves incident to a vertex in 
$v_1, \ldots, v_k$.
Then $k'(S)\le k$ and 
 $
2h(T)\ge  |S| - k'(S) \ge (n - |V(P)|) - k   \ge   n - 2|V(P)| \ge \frac{n}{3}. 
$\\

 \noindent
 Let $T$ be a haircomb with $m$ vertices on its spine 
and $k$ spinal vertices, where $k\le m$.  
We denote the length of the $i$th leg of $T$ with $l_i$. 
Since $(k+1)\max(l_1,\dots,l_k,m) \ge m+\sum_1^k l_i = n$, 
either $k+1 \ge \sqrt{n}$ implying $m\ge k+1 \ge \sqrt{n}$ or 
$l_i\ge \sqrt{n}$ for some $i$, which implies that $\diam(T)\ge \sqrt{n}$. 
By Theorem \ref{diam-kg}, $2h(T) \ge \sqrt{n}$. 
\end{proof}

\begin{proof}[Proof of Theorem~\ref{spiders}]
~\\
Let $m$ be an integer such that $k=2m$ or $k=2m+1$.
Using Construction \ref{k-legs}, we see that 
\begin{equation}\label{spdr-bd}
2h(G) \geq n- \left[ (l_1-l_2) + (l_3-l_4 ) + \ldots + (l_{2m-1}-l_{2m}) + x +1\right] \geq  n- l_1+l_{2m}-x-1,
\end{equation}
where $x= 0$ if $k$ is even and $x= l_k$ if $k$ is odd. 
This observation is used in the following cases. 

\begin{itemize}
\item{}  Let  $k=3$. Let $\min(l_1 - l_2, l_2 - l_3) = cn$ for some $c \ge 0$, where $l_1\ge l_2\ge l_3.$ 
By Construction \ref{three-legs},  we have $ 2h(T) \ge n - cn$.
Without loss of generality,  assume that $l_1 - l_2  =cn$. 
 Then  $l_1 = l_2 + cn$ and $l_2 \ge l_3 + cn$ by our assumption and therefore, 
$ n = l_1 + l_2 + l_3 +1 \ge 3l_3 + 3cn$. 
This implies that $l_3\le (n - 3cn)/3$ and by Lemma \ref{lem-main}, 
$2h(T) \ge l_1 + l_2 + 1 = n - l_3  \ge 2n/3 + cn$. 
Thus $$2h(\mathcal{S}_{n,3}) \geq 
\min_{0\leq c\leq 1} \max\left\{n - cn,  \frac{2n}{3} + cn \right\} 
\ge 5n/6.$$

\item{} Let  $k=4$.   We have  $ 2h(T) \ge \diam(T)+1\ge l_1 + l_2+1$ and 
\eqref{spdr-bd} provides that $2h(T)\ge n - (l_1 - l_4) - 1$. 
Adding these inequalities gives that 
$4h(T) \ge n + l_2 + l_4. $ 
By letting $l_2 + l_4 = c'n$, for some $c'$, $0\leq c'\leq 1$, 
we rewrite this bound as
$4h(T) \ge (1+c')n .$ 
On the other hand, 
$2h(T) \ge \diam(T) +1 \ge l_1 + l_3 + 1 = (1 - c')n$. 
Thus, 
$$2h(\mathcal{S}_{n,4}) \geq \min_{0\leq c'\leq 1} 
\max\left\{\frac{(1+c')n}{2},  (1 - c')n\right\}\ge 2n/3.$$

\item{} Let $k \geq 5$. Let $T$ be a spider on $n$ vertices with $k$ legs 
having  $\ell_1\geq \ell_2 \geq \cdots \geq \ell_k$ vertices, respectively,  and the head  $v$.  
Observe that $l_2\ge (n-l_1-1)/(k-1)$.  Assume that $2h(T) = cn$ for some $c>0$. 
By Lemma \ref{lem-main},  $2h(T) = cn \ge \diam(T)+1 \ge l_1+l_2+1\ge l_1 + (n-l_1-1)/(k-1) + 1.$
Thus, $-l_1 \ge n(1-ck+c)/(k-2)+1.$  Using \eqref{spdr-bd},  
$cn = 2h(T) \ge n-l_1-1 \ge n +  \frac{n(1-ck+c)}{k-2}.$
Since $k$ is fixed, this implies that $c\ge 1/2 + 3/(4k-6),$ i.e., $h(T) \ge n/4 + 3n/(8k-12).$

The bound $h(\mathcal{S}_{n,k}) \geq (n+2)/4$ for $k\ge 5$ 
is attained by a spider  $T$ with $n/2$ single-vertex  legs,  and one leg, $P$,  with $n/2-1$ vertices. 
If $H$ is an empty graph on  $n/2$ vertices, then $T$ is an $(H,   n/2-1, v )$-flower, where $v$ is the head of $T$. 
 Let $S_1, S_2$ be homometric sets of $T$.
Then by Lemma~\ref{flower}, 
$S_1\cup S_2 \subseteq V(H) \cup \{v, v_1\}$ or $ S_1\cup S_2 \subseteq V(P)\cup \{v, u\}$, where $u\in V(H)$, $v_1\in V(P)$ is adjacent to $v$. 
In the first case, we can easily see that $v \not\in S_1\cup S_2$, so  $2h(T) \leq n/2+1$. 
\end{itemize}
\end{proof}


\begin{proof}[Proof of Theorem~\ref{dense-case}]
In a graph with diameter at most $2$, 
any two distinct vertices are at  distance   $1$ or $2$. 
Therefore, Theorem~\ref{num-edges} implies this result.  
\end{proof}

%
%



%

\begin{proof}[Proof of Theorem~\ref{general-diam2}]

 First, we give a fact observed in \cite{caroyuster} stating that 
 if $A\subseteq V(G)$ and $\sum_{v \in A }\deg(v) = \sum_{v \not\in A} \deg(v)$ then 
$|E(G[A])|=|E(G[V-A])|$.   To see this,  note that 
$|E(G[A])|=(1/2)  ((\sum_{v \in A} \deg(v)) - |E(A, V-A)|)$ and 
$|E(G[V-A])|=(1/2) ((\sum_{v\in V- A} \deg(v))  - |E(A, V-A)|).$ 
So, the difference of these two numbers is 
$ \frac{1}{2} [\, \sum_{v \in A }\deg(v) -\sum_{v \not\in A} \deg(v)\,] = 0$.

\noindent
Let $\mathcal{D} = \{0,1,\dots,\Delta(G)\}$
and $i \in \mathcal{D}$. 
Let    
$$V_i = \{v\in V: \deg(v)=i\}, \quad V_i = A_i\cup B_i \cup S_i,$$ 
where $A_i, B_i, S_i$ are disjoint,  $|S_i|\leq 1$ and $|A_i|=|B_i|$.  
I.e., split $V_i$ into two equal parts if possible, and otherwise, 
put a remaining vertex into a set $S_i$.

\noindent
 Let $$A' = \bigcup _{i \in \mathcal{D}} A_i,
\quad  B'= \bigcup_ {i \in \mathcal{D}} B_i, 
\quad S= \bigcup_{i \in \mathcal{D}} S_i.$$
 So, $V(G) = A'\cup B'\cup S$.  Let $s(G) =|S|$. 
 We shall show that there are disjoint  subsets $A, B$ of $n/2$ vertices each, such that 
$\sum_{v\in A} \deg(v) = \sum_{v\notin A} \deg(v)$, i.e.,  
$A$ and $B$ induce the same number of edges in $G$.

If $|V_i|$ is even for each $i$, let $A= A'$ and $B= B'.$ 

Now, assume that $m=|\{i: |V_i|\, \text{ is odd}\}|>n/2.$ 
Note that all vertices in $S$ have distinct odd degrees $a_1, a_2, \cdots, a_m $, 
say $1\leq a_1 <a_2 <\cdots < a_m <n$, and since $m> n/2$, $n<2m.$ 
Note also that $m$ is even, since the total number of vertices of odd degree is even. 
So, we could apply Theorem~\ref{kar}  
and split $\{a_1, \ldots, a_m\}$ in two parts, 
$U$ and $U'$  of equal sizes and with almost equal sums. 
Since 
$\sum _{v\in V} \deg(v) = \sum _{v\in A'} \deg(v)  + 
\sum _{v\in B'} \deg(v) + \sum_{v\in S} \deg(v) = 
2\sum_{v\in A'} \deg(v) + \sum_{v\in S} \deg(v)$, and 
this degree-sum is even, it follows that $\sum_{v\in S} \deg(v)$ is even. 
Thus $\sum_{i=1}^{m} a_i =\sum_{v\in S} \deg(v)$ is even, 
and the sum of elements in $U$ is exactly equal to the sum of elements in $U'$. 
Let 
$$A'' = \{ v\in S: \deg(v) \in U\},\quad B'' = \{v\in S: \deg(v) \in U'\}.$$ 
Finally, let $A= A' \cup A''$ and $B = B'\cup B''.$ 
\end{proof}




\section{More results on trees}\label{More-Trees}

For a tree $T$, let $d_i=d_i(T)$ be the number of vertices of degree $i$.

\begin{lem}\label{bound-sib}
~\\
a) $ d_1 - bad(T) \ge 2+ \sum_{i\ge 4} d_i$,\\
b)  $d_1 \ge 2/3 ( n - d_2 - d_3)$,  and \\
c) $\sum_{i\ge 4} d_i \le (1/3) (n - d_2 - d_3)$. 
\end{lem}
\begin{proof}
Observe first that $d_1 = 2+ \sum_{i\ge 3}d_i(i-2)$.  
Thus $d_1 \ge 2 + \sum_{i\ge 4} 2d_i.$

Since bad vertices have degree at least $3$,  we have 
that  $bad(T) \le  d_3 + \sum_{i\ge 4} d_i$, so $ d_1 - bad(T) \geq   2+ \sum_{i\ge 4} d_i$.   
We also have that  $d_1\ge \sum_{i\ge 4} 2d_i = 2(n - d_1 - d_2 - d_3)$,  thus  $d_1 \ge 2/3 ( n - d_2 - d_3).$
To show the last inequality, observe that $\sum_{i\ge 4} d_i = n- d_1-d_2-d_3$, so using (b), 
$\sum_{i\ge 4} d_i = n- d_1-d_2-d_3 \geq n- 2/3 ( n - d_2 - d_3) -(d_2 +d_3) = 1/3 ( n - d_2 - d_3).$
\end{proof}

\begin{thm}\label{general-trees} 
Let $T$ be a tree on $n$ vertices. 
Then  $h(T) = \Omega(\sqrt{n})$ if one of the following holds: \\
a) $bad(T) = o(\sqrt{n})$, \\
b) $bad_3(T) =o(\sqrt{n})$ and $bad(T)= \Omega(\sqrt{n})$, \\
c) $bad_l(T) = o(n)$.\\
Moreover $h(T) \geq \max\{(n- d_2 - 4 d_3)/6, (n-d_1-d_2)/(\diam(T)+1)\}.$
 \end{thm}


\begin{com}
 There are trees on $n$ vertices, such as double haircombs, 
for which our proof techniques do not provide 
bounds of the order of magnitude $\sqrt{n}$. 
A {\it double haircomb}  is constructed  from a haircomb  by attaching 
vertex-disjoint paths to the vertices of the legs, see Figure~\ref{dhcomb}.  
By appropriately choosing the distances between the vertices of degree $3$, 
one can construct an $n$-vertex double haircomb 
$T$, such that  $bad_3(T)   \geq c\sqrt{n}$,  $bad_l(T)\geq c'n$ for some constants $c, c'>0$ and 
$\diam(T) = o(\sqrt{n}).$ 
 \end{com}

\begin{figure}[h]
\begin{center}
\vspace{-.4cm}
\includegraphics[scale=0.4]{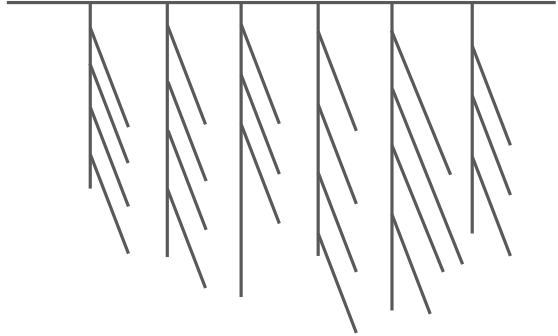}
\vspace{-2.4cm}
\caption{{\small A double haircomb.}}
\label{dhcomb}
\end{center}
\end{figure}

\begin{proof}[Proof of Theorem~\ref{general-trees}]

To prove the first part of the theorem, we assume throughout the proof that $\diam(T) = o(\sqrt{n})$, 
otherwise by Lemma \ref{lem-main}, $h(T) =\Omega(\sqrt{n})$.  

\begin{itemize}
\item{}  $bad(T) = o(\sqrt{n}).$ 
By Lemma~\ref{lem-main},  
$2h(T) \ge n/\diam(T) - bad(T) = \Omega(\sqrt{n})$.

\item{}  $bad_3(T) = o(\sqrt{n})$ and $bad(T)=\Omega(\sqrt{n}).$ \\
By Lemma \ref{bound-sib}, $d_1 - bad(T)\ge \sum_{i\ge 4} d_i \ge bad(T) - bad_3(T) = \Omega(\sqrt{n}).$ 
Lemma  \ref{lem-main} implies that $2h(T) =  \Omega(\sqrt{n})$.

\item{} $bad_l(T) = o(n).$ 
By Lemma~\ref{lem-main}, 
$h(T) \ge (n - bad_l(T))/\diam(T) = \Omega(\sqrt{n})$. 
\end{itemize}

To prove  the second part of the theorem,  we use Lemma~\ref{lem-main} and 
Lemma \ref{bound-sib}-(b),(c). 
We have that 
$2h(T)\ge d_1 - bad(T) \ge (d_1 -\sum_{i\ge 4} d_i) - d_3 
\ge 1/3 (n - d_2 - d_3) - d_3 = 1/3 (n - d_2 - 4d_3) $.

Let $S$ be the set of vertices with degree at least $3$ in $T$,   $|S| = n- (d_1+d_2)$. 
Let $r$ be a leaf in $T$, $V(T) = N_0 \cup N_1\cup N_2 \cup \dots \cup N_t$, 
where $N_0 = \{r\}$, $N_i=N_i(r)$ and $t = \lceil \diam(T)/2 \rceil$. 
There is an $i\in [t]$  such that $|N_i(r)\cap S|  \ge |S|/t$. 
Let $L=N_i(r)\cap S$.   
We have $|L| \ge (n - d_1 - d_2)/t \ge 2(n-d_1-d_2)/(\diam(T)+1).$ 
Let the set of children of $L$ with respect to $P(T,r)$ be $L'$.  
Since each vertex in $S$ and thus in $L$ has degree at least three,  $|L'|\ge 2|L|$. 
Using Lemma~\ref{lem-main} and the fact that $k'(L')\le |L|$, 
we have  $h(T)\ge (|L'| - k'(L'))/2 \ge (2|L| - |L|)/2 = |L|/2\ge (n-d_1-d_2)/(\diam(T)+1).$
\end{proof}

\section{Concluding Remarks}

The definition of homometric sets allows for little control over what pairs of vertices realize what distance.
This makes proving the upper bounds on $h(n)$ difficult.  One may consider another definition for two sets being homometric.  
For a graph $G$,  let  $K(G)=(V(G), c)$  be a complete graph on vertex set $V(G)$ with an edge-coloring $c$, where 
$c(u,v)$ is equal to the distance between $u$ and $v$ in $G$.  
Let two disjoint sets $S_1,S_2\subset V(G)$ be {\it similar} if 
there is an isomorphism between $K(G)[S_1]$ and $K(G)[S_2]$. 
Note that there may be homometric sets that are not similar. 
Let $T$ be a spider with four long legs of equal length, let 
 $S_1$ consist of four vertices on distinct legs with distances 2,2,2,6
to the head of the spider, respectively, 
let $S_2$ be the set of four vertices - one is the head, three other are on three distinct
legs with distance $4$ to the head.  Although $D(S_1) = D(S_2)= \{4,4,4, 8,8,8\}$, 
$S_1$ and $S_2$ are not similar, since  there is a triangle  with all edges colored $4$ in $K(G)[S_1]$ and
there is a triangle with all edges  colored $8$ in $K(G)[S_2]$.

\section{Acknowledgements}
The authors thank Michael Young 
for pointing out one of the special cases in Theorem~\ref{spiders}.

\bibliographystyle{amsplain}
\bibliography{hom-sets-11-24-11}

\providecommand{\bysame}{\leavevmode\hbox to3em{\hrulefill}\thinspace}
\providecommand{\MR}{\relax\ifhmode\unskip\space\fi MR }
\providecommand{\MRhref}[2]{%
  \href{http://www.ams.org/mathscinet-getitem?mr=#1}{#2}
}
\providecommand{\href}[2]{#2}
\begin{thebibliography}{1}

\bibitem{APY}
M.~Albertson, J.~Pach, and M.~Young, \emph{Disjoint homometric sets in graphs},
  submitted.

\bibitem{bollobas}
B.~Bollob{\'a}s, \emph{Modern graph theory}, Graduate Texts in Mathematics,
  vol. 184, Springer-Verlag, New York, 1998.

\bibitem{caroyuster}
Y.~Caro and R.~Yuster, \emph{Large disjoint subgraphs with the same order and
  size}, European J. Combin. \textbf{30} (2009), no.~4, 813--821.

\bibitem{karolyi}
G.~K{\'a}rolyi, \emph{Balanced subset sums in dense sets of integers}, Integers
  \textbf{9} (2009), A45, 591--603.

\bibitem{LSS}
P.~Lemke, S.~S. Skiena, and W.~D. Smith, \emph{Reconstructing sets from
  interpoint distances}, Discrete and Comput. Geom., Algorithms Combin.,
  vol.~25, Springer, Berlin, 2003, pp.~507--631.

\bibitem{lovasz}
L.~Lov{\'a}sz, \emph{Kneser's conjecture, chromatic number, and homotopy}, J.
  Combin. Theory Ser. A \textbf{25} (1978), no.~3, 319--324.

\bibitem{matousek}
J.~Matou{\v{s}}ek, \emph{Thirty-three miniatures}, Student Mathematical
  Library, vol.~53, American Mathematical Society, Providence, RI, 2010,
  Mathematical and algorithmic applications of linear algebra.

\bibitem{RS}
J.~Rosenblatt and P.~D. Seymour, \emph{The structure of homometric sets}, SIAM
  J. Algebraic Discrete Methods \textbf{3} (1982), no.~3, 343--350.

\bibitem{west}
D.~B. West, \emph{Introduction to graph theory}, Prentice Hall Inc., Upper
  Saddle River, NJ, 1996.

\end{thebibliography}

\end{document}